\documentclass[12pt]{amsart}

\usepackage{amssymb}

\makeatletter
\@addtoreset{equation}{section}
\makeatother

\theoremstyle{plain}
\newtheorem{theorem}{Theorem}[section]
\newtheorem{lemma}[theorem]{Lemma}

\newtheorem{corollary}[theorem]{Corollary}
\newtheorem{question}[theorem]{Question}
      
\theoremstyle{definition}

\newtheorem{remark}[theorem]{Remark}
\newtheorem{example}[theorem]{Example}

\newcommand{\Ree}{{\mathbb R}}
\newcommand{\En}{{\mathbb N}}
\newcommand{\Que}{{\mathbb Q}}

\newcommand{\Del}{\Delta}
\newcommand{\eps}{\varepsilon}
\newcommand{\lam}{\lambda}

\newcommand{\fA}{\mathcal{A}}
\newcommand{\fB}{\mathcal{B}}
\newcommand{\fH}{\mathcal{H}}
\newcommand{\fN}{\mathcal{N}}

\newcommand{\spn}{\mathrm{span}}
\newcommand{\ind}{\mathrm{ind}}

\newcommand{\op}{\mathrm{op}}

\newcommand{\alg}{\mathrm{G}}
\newcommand{\cstarr}{\mathrm{C}^*_r}
\newcommand{\fal}{\mathrm{A}}
\newcommand{\fsal}{\mathrm{B}}
\newcommand{\bl}{\mathrm{L}}
\newcommand{\posbltwo}{\mathrm{PL^2}}
\newcommand{\str}{\mathrm{S}_r}
\newcommand{\trstr}{\mathrm{TS}_r}

\newcommand{\LUC}{\mathcal{LUC}}  
\newcommand{\SL}{\mathrm{SL}}  
\newcommand{\alp}{\alpha}
\newcommand{\sig}{\sigma}

\begin{document}
 
\author{Brian E. Forrest}
\address{Pure Mathematics, University of Waterloo, Waterloo, Ontario, N2L 3G1,Canada}
\email{beforrest@uwaterloo.ca}
 
\author{Nico Spronk}
\address{Pure Mathematics, University of Waterloo, Waterloo, Ontario, N2L 3G1, Canada}
\email{nspronk@uwaterloo.ca}

\author{Matthew Wiersma}
\address{Mathematical and Statistical Sciences, University of Alberta, Edmonton, Alberta, T6G 2G1, Canada}
\email{mwiersma@ualberta.ca}

\title[Existence of tracial states]{Existence of tracial states on reduced group C*-algebras}

\begin{abstract}
Let $G$ be a locally compact group.  It is not always the case that its reduced C*-algebra
$\cstarr(G)$ admits a tracial state.  We exhibit closely related necessary and sufficient conditions
for the existence of such.  We gain a complete answer when $G$ compactly generated.
In particular for $G$ almost connected, or more generally when $\cstarr(G)$ is nuclear,
the existence of a trace is equivalent to amenability.  We  exhibit two examples of classes of totally 
disconnected groups for which $\cstarr(G)$ does not admit a tracial state.
\end{abstract}


\thanks{The first two authors were partially supported by an NSERC Discovery Grants.}


 \date{\today}
 
 \maketitle
 

\section{Introduction and background}


If $G$ is a discrete group, then it is well-known that its reduced C*-algebra $\cstarr(G)$ admits a tracial state.
The condition of when this trace on $\cstarr(G)$ is unique is linked to the simplicity $\cstarr(G)$. Indeed, Kalantar and Kennedy \cite{kalantark} showed that the trace on $\cstarr(G)$ is unique when $\cstarr(G)$ is simple and a complete characterization of when $\cstarr(G)$ admits a unique tracial state for a discrete group $G$ was solved by Breuillard, Kalantar, Kennedy and Ozawa \cite{breuillardkko} and Haagerup \cite{haagerup} by using techniques developed for studying when $\cstarr(G)$ is simple.

The study of simplicity for $\cstarr(G)$ has recently migrated to the case when $G$ is a totally disconnected group due to significant work of Raum \cite{raum1,raum2} and Suzuki \cite{suzuki}. In his example, Suzuki showed that that if $\cstarr(G)$ admits a tracial state, then it must be unique.

We focus our attention on the following question.

\begin{question}\label{question}
What conditions on a locally compact group $G$ characterize when $\cstarr(G)$ admits a tracial state.
\end{question}

Ng \cite{ng} considered this question from the perspective of finding a condition which may be added to the 
assumption that $\cstarr(G)$ is nuclear to determine that $G$ is amenable.  He showed that this 
extra assumption is the existence of a tracial state.
We give an alternate proof to Ng's result in Section \ref{sec:nuclear}.
He used a theory of strictly amenable representations of C*-algebras; we use only one aspect of the theory
of amenable traces.

More generally we answer this Question \ref{question} fully when $G$ is compactly generated, 
in particular when $G$ is almost connected or when $G$ has property (T).  This is given in Section \ref{sec:main}.

\subsection{Background}
The {\it reduced C*-algebra} of $G$ is the C*-algebra generated by the left regular representation
on the $\bl^2$-space with respect to a left Haar measure:
$\lam_G:G\to \fB(\bl^2(G))$ ($\lam_G$ will be denoted $\lam$ when the group is unambiguous),
$\lam(s)g(t)=g(s^{-1}t)$ for $s$ in $G$, $g$ in $\bl^1(G)$ and a.e.\ $t$ in $G$.  
We have an integrated form $\lam:\bl^1(G)\to \fB(\bl^2(G))$, $\lam(f)g(t)=f\ast g(t)=\int_G f(s)g(s^{-1}t)\,dt$,
where this integral may be understood in the weak operator topology.  Then
$\cstarr(G)=\overline{\lam(\bl^1(G))}$, the norm closure in $\fB(\bl^2(G))$.

We shall use the duality $\cstarr(G)^*\cong \fsal_r(G)$ established by Eymard \cite{eymard}.
Let us briefly recall the definition of the {\it reduced Fourier-Stieltjes algebra} $\fsal_r(G)$.
By Godement \cite{godement} the cone $\posbltwo(G)$ of square-integrable positive
definite functions consists of certain elements of the form $\langle \lam(\cdot)g,g\rangle=\bar{g}\ast \check{g}$ where
$g\in\bl^2(G)$,  and $\check{g}(t)=g(y^{-1})$ for a.e.\ $t$ in $G$.  Furthermore, for any
$u$ in $\spn \posbltwo(G)\subset \bl^\infty(G)$ we have
\[
\sup\left\{\left| \int_G f(s)u(s)\, ds\right|:f\in\bl^1(G),\,\|\lam(f)\|\leq 1\right\}<\infty
\]
so $\spn \posbltwo(G)$ forms a subspace of $\cstarr(G)^*$, which is, in fact weak* dense; we denote this
space $\fsal_r(G)$.  The space may also be realized as the space of all matrix coefficients
$\langle \pi(\cdot)\xi,\eta\rangle$ where $\pi:G\to\fB(\fH)$ is a continuous unitary representation
which is weakly contained in $\lam$, and $\xi,\eta\in\fH$.
In particular, the norm closure $\fal(G)$ of $\spn \posbltwo(G)$ in $\fsal_r(G)$ is called the {\it Fourier algebra},
and may be realized all matrix coefficients of the left regular representation:
functions of the form $\langle \lam(\cdot)g,h\rangle=\bar{h}\ast \check{g}$ where
$g,h\in\bl^2(G)$.  We let
\[
\str(G)=\{u\in\fsal_r(G):u\text{ is positive definite and }u(e)=1\}
\]
which is the state space of $\cstarr(G)$, i.e.\ each $u$ corresponds to a positive norm one functional.

A state is called {\it tracial} if $\int_G f\ast f'(s)u(s)\, ds=\int_G f'\ast f(s)u(s)\, ds$ for each
$f,f'$ in $\bl^1(G)$.  By left invariance of the measure, we see that this is equivalent to having
$u(ts)=u(st)$ for each $s,t$ in $G$.  Hence we consider the set of tracial states, given by
\[
\trstr(G)=\{u\in\str(G):u(tst^{-1})=u(s)\text{ for all }s,t\text{ in }G\}.
\]
Our goal is to study when this set is non-empty.

The following observations are well-known.

\begin{lemma}\label{lem:restriction}
{\bf (i)} Let $H$ be a closed subgroup of $G$.  If there is $u$ in $\fsal_r(G)$ for which
the restriction $u|_H=1$, then $H$ is amenable.

{\bf (ii)} Let $N$ be an amenable normal subgroup of $G$.  Then
$\fal(G/N)$ is isomorphic to a subalgebra of elements in $\fsal_r(G)$.
\end{lemma}

\begin{proof}
(i) The restriction theorem (\cite{herz,arsac,delaported})
tells us that $\fal(G)|_H=\fal(H)$.  Hence weak* density of Fourier algebras in reduced
Fourier algebras provides that $\fsal_r(G)|_H\subseteq \fsal_r(H)$.  In particular, we see that
$1=u|_H\in \fsal_r(H)$, which tells us that $H$ is amenable.

(ii)  Let $\lam_{G:N}:G\to\fB(\bl^2(G/N))$ denote the
left quasi-regular representation of $G$, $\lam_{G:N}(s)g(tN)=g(s^{-1}tN)$ for a.e.\
$tN$ in $G/N$.  Then amenability of $N$ and
Fell's continuity of induction \cite{fell} gives weak containment relation
\[
\lam_{G:N}=\ind_N^G 1_N\prec \ind_N^G \lam_N=\lam_G
\]
and it follows that $\fal(G/N)$, which is identifiable with the space of matrix coefficients
of $\lam_{G:N}$, embeds naturally as a closed subspace of $\fsal_r(G)$.  (This is in the nature of
the functorial results of Arsac \cite{arsac}.)  We also remark that his result is also observed in \cite{leess}.
\end{proof}

\section{Main result}\label{sec:main}

We say that $G$ is an {\it invariant neighbourhood group}, or [IN]-group, if there exists
a relatively compact neighbourhood $W$ of $e$ for which $xWx^{-1}=W$ for all $x$ in $G$;
and a {\it small invariant neighbourhood group}, or [SIN]-group if $G$ admits a neighbourhood basis
of invariant neigbourhoods.
We observe that the class of [SIN]-groups contains discrete groups, abelian groups and compact groups.
Both classes are closed under products.  The class of [IN]-groups is closed under 
extension when the normal subgroup is compact.
See the discussion in the book of Palmer \cite[12.6.15 \& 12.6.16]{palmer}.  
Moreover, \cite[12.1.31]{palmer} shows that any
[IN]-group $G$ admits a compact normal subgroup $K$ by which $G/K$ is [SIN] -- we write [IN]=[K$^\text{SIN}$].
Any almost connected [SIN] group is of the form $V\rtimes K$ where $V$ is a vector group and $K$
has a finite action of $V$, and is hence amenable, hence so too is any almost connected [IN]-group.

\begin{theorem}\label{theo:main}
Consider the following conditions on $G$:

{\bf (i)} $\cstarr(G)$ admits a tracial state, i.e.\ $\trstr(G)\not=\varnothing$;

{\bf (ii)} $G$ admits a normal amenable closed subgroup $N$ for which $G/N$ is
an [IN]-group ---  we say $G$ is of class [Am$^\text{IN}$];

{\bf (ii')} $G$  admits a normal amenable closed subgroup $N$ for which each 
compactly generated open subgroup of $G/N$ is an [SIN]-group.

Then (ii) $\Rightarrow$ (i) $\Rightarrow$ (ii').  In particular, if $G$ is compactly generated, then
(i) $\Leftrightarrow$ (ii).
\end{theorem}

See Remark \ref{rem:suzuki} below to see that (ii') does not imply (i), in general.  Also, comments
above show that [Am$^\text{IN}$]=[Am$^\text{SIN}$].

\begin{proof} {\bf (ii)} $\Rightarrow$ {\bf (i)}.  Lemma \ref{lem:restriction} (ii)
shows that $\fal(G/N)$ identifies as a subalgebra of $\fsal_r(G)$.
Let $W$ be a compact neighbourhood of $G/N$ which is invariant under inner automorphisms.
We let for $sN$ in $G/N$
\[
u(sN)=\langle \lam_{G/N}(sN)1_W,1_W\rangle=m((sN)W\cap W)
\]
which defines an element of $\posbltwo(G/N)\subseteq\fal(G/N)$, where $m$ denotes the Haar measure.  
If $t\in G$ we have
\[
u(tst^{-1}N)=m((tst^{-1}N)W\cap W)=m(tN[(sN)W\cap W]t^{-1}N)
\]
since $(tN)W(t^{-1}N)=W$.  
Hence $s\mapsto \frac{1}{m(W)}u(sN)$ defines an element of $\trstr(G)$.

{\bf (i)}$\Rightarrow${\bf (ii')}.  Let $u\in\trstr(G)$.  We may replace $u$ by $|u|^2=u\bar{u}$
in $\trstr(G)$, so we may assume that $0\leq u\leq 1$.  Any set
$S=u^{-1}(B)$, where $\varnothing\not= B\subseteq [0,1]$ is necessarily
satisfies $sSs^{-1}=S$ for all $s$ in $G$.

Let $N=u^{-1}(\{1\})$.  It is well-known that $N$ is a subgroup.  Indeed, write
$u=\langle \pi(\cdot)\xi,\xi\rangle$ where $\pi:G\to\fB(\fH)$ is a unitary representation and
$\|\xi\|=1$ in $\fH$, and uniform convexity provides that $N=\{s\in G:\pi(s)\xi=\xi\}$.
Hence $N$ is a normal subgroup with $u$ constant on cosets of $N$.  
Lemma \ref{lem:restriction} (i) tells us that $N$ is amenable.

For ease of notation, let us now replace $G$ by $G/N$ and assume that
\begin{equation}\label{eq:assumption}
u^{-1}(\{1\})=\{e\}.
\end{equation}
Let $H$ be a compactly generated open subgroup of $G$, so there is a symmetric
compact neighbourhood $K$ of $e$ for which $H=\bigcup_{n=1}^\infty K^n$.
Let $0<\eps<1$ be such that
\[
u^{-1}((1-\eps,1])\cap K^3\subset (K^3)^\circ\text{ (interior).}
\]
If $K$ is open, $\eps$ may be chosen arbitrarily; if not (\ref{eq:assumption}) provides $\eps$ so $u(s)<1-\eps$ for
$s$ in $\partial(K^3)$ (boundary).  For natural numbers $n$ for which $n>1/\eps$ we let
\[
V_n=u^{-1}((1-\tfrac{1}{n},1])\text{ and }W_n=V_n\cap K^3.
\]
If $U$ is an open neighbourhood of $e$, let $n$ be so $u(s)<1-\frac{1}{n}$ for $s$ in $K^3\setminus U$
--- (\ref{eq:assumption}) provides that we may find $\frac{1}{n}<\min_{s\in K^3\setminus U} [1-u(s)]$ ---
and we see that $W_n\subseteq U$.  Hence $\{W_n\}_{n=\lceil 1/\eps\rceil}^\infty$ is a neighbourhood base for
$e$.  In particular, for sufficienlty large $n$, say $n\geq n_0\geq \lceil 1/\eps\rceil$, we have $W_n\subseteq K$.  
Then for such $n$, the symmetry of $K$ provides that
\[
\bigcup_{s\in K}sW_ns^{-1}\subseteq KW_nK\subseteq K^3.
\]
But also
\[
\bigcup_{s\in G}sW_ns^{-1}\subseteq \bigcup_{s\in G}sV_ns^{-1}=V_n
\]
so we conclude that
\[
\bigcup_{s\in K}sW_ns^{-1}\subseteq V_n\cap K^3=W_n
\]
whence $sW_ns^{-1}=W_n$ for each $s$ in $K$.  It is immediate that $sW_ns^{-1}=W_n$ for each 
$s$ in $H=\bigcup_{n=1}^\infty K^n$.  Hence $\{W_n\}_{n=n_0}^\infty$ shows that $H$ is a [SIN]-group.
\end{proof}

For separable groups, the following is proved by Ng \cite{ng}, using much more indirect techniques.

\begin{corollary}
If $G$ is almost connected, then $\cstarr(G)$ admits a tracial state if and only if $G$ is amenable.
\end{corollary}

\begin{proof}
This follows immediately from the theorem, above, and the fact that almost connected groups
are compactly generated and that almost connected [IN]-groups are
amenable, thanks to the structure results indicated at the beginning of this section.
\end{proof}

To put Theorem \ref{theo:main} in context, we shall use the following observation.

\begin{lemma}\label{lem:INcns}
Let $G$ be a totally disconnected [IN]-group.  Then $G$ admits an open compact normal subgroup.
\end{lemma}

\begin{proof}
As noted at the beginning of this section, there is a compact normal subgroup $K$ for which $G/K$ is a [SIN]-group.
Then $G/K$ is totally disconnected and there is an open compact subgroup $M$, and an invariant
neighbourhood $W$, so $K\subseteq W\subseteq M$.  We again use total disconnectivity of $G/K$ to
find an open subgroup $L$ with $K\subseteq L\subseteq W$.  The the subgroup generated by
$\bigcup_{x\in G}xLx^{-1}$ is open and normal, but contained in $M$ and hence compact.
\end{proof}

Actually, the above lemma establishes that a totally disconnected [SIN]-group is {\it pro-discrete}, i.e.\
admits arbitrarily small compact open normal subgroups.

\begin{corollary}
The following three classes of locally compact groups coincide

{\bf (a)} [AM$^\text{IN}$]\qquad\qquad\qquad {\bf (b)} [AM$^\text{SIN}$]

{\bf (c)} groups admitting an open normal amenable subgroup.
\end{corollary}

\begin{proof}
Clearly (c) implies (b) implies (a).  Let us see that (a) implies (c).  If $G$ is in [AM$^\text{IN}$],
and $N$ is a normal amenable subgroup for which $G/N$ is [IN].  Then the connected component of the identity
$(G/N)_e$ is a connected [IN]-group, hence amenable, and $(G/N)/(G/N)_e$,
is a totally disconnected [IN]-group. Thus Lemma \ref{lem:INcns} provides an open normal compact 
subgroup of  $(G/N)/(G/N)_e$, whose pre-image with respect to
the quotient map $q:G\to (G/N)/(G/N)_e$ is the desired open normal amenable subgroup.
\end{proof}

\begin{remark}\label{rem:suzuki}
Condition (ii') does not guarantee the existence of a tracial state on $\cstarr(G)$.

Suzuki \cite{suzuki} considers totally disconnected groups of the form 
$G=\bigoplus_{n\in\En}\Gamma_n\rtimes\prod_{n\in\En}F_n$
where each pair $(\Gamma_n,F_n)$ consists of a discrete group and a finite group acting on
on it for which $\cstarr(\Gamma_n\rtimes F_n)$ admits unique tracial state.  
Here $F=\prod_{n\in\En}F_n$ is equipped with 
product topology, and the action of $F$ on $\Gamma=\bigoplus_{n\in\En}\Gamma_n$ is component wise.
Each compactly generated open subgroup of $G$ is contained in one
of the form $H=\Lambda\rtimes K$ where $\Lambda
\subseteq \bigoplus_{j=1}^m\Gamma_{n_j}$ 
and $K$ is an open subgroup of $F$.  Then
$K\cap\prod_{n\in\En\setminus\{n_1,\dots,n_m\}}F_n$ is an open, 
compact normal subgroup of $H$, thereby showing that $H$
is an [IN]-group.  Thus (ii') of Theorem \ref{theo:main} holds.  

The sequences $K_n=\prod_{k=n+1}^\infty F_n$ and
$L_n=\bigoplus_{k=1}^n\Gamma_k\rtimes F$ satisfy that each $K_n\lhd L_n$, and with each
$\cstarr(L_n/K_n)\cong \bigotimes_{\min,k=1}^n\cstarr(\Gamma_k\rtimes F_k)$ admitting a unique tracial state.
On $\lam(\xi_{K_n})\cstarr(L_n)\cong\cstarr(L_n/K_n)$, where $\xi_{K_n}=\frac{1}{m(K_n)}1_{K_n}$,
this state is given by
\[
u_n(x)=\frac{1}{m(K_n)}\langle\lam(x)1_{K_n},1_{K_n}\rangle=\frac{m(xK_n\cap K_n)}{m(K_n)}\text{ for }x\text{ in }L_n.
\]

Now suppose there is $u$ in $\trstr(G)$.  Then $\xi_{K_n}\ast u\ast\xi_{K_n}$ is a sequence of positive definite 
elements converging in $\fsal_r(G)$-norm to $u$, hence converging uniformly.  Hence
we have uniform on compact convergence 
\[
u=\lim_n \xi_{K_n}\ast u\ast\xi_{K_n}|_{L_n}=\lim_n u_n.
\]
But on any fixed $L_{n_0}$,  the latter sequence converges pointwise to $1_{\{e\}}$, which is
not continuous as $G$ is not discrete.  Thus $\trstr(G)=\varnothing$.
\end{remark}

\begin{remark}\label{rmk}
It curious to note how some conditions in Theorem \ref{theo:main} relate to {\it inner amenability},
the property that $\bl^\infty(G)$ admits a conjugation-invariant mean.

{\bf (i)}  If $G$ admits the property that each compactly generated open subgroup $H$ is [IN], then $G$
is inner amenable.  

Indeed, for each such $H$ let $W_H$ be a relatively compact invariant neighbourhood.
Any cluster point in $\bl^\infty(G)^*$ of the net of elements $\xi_H=\frac{1}{m(W_H)}1_{W_H}$, indexed over increasing 
$H$, is a conjugation invariant mean.  

Moreover, if we can arrange that $W_H\searrow\{e\}$, i.e.\
eventually enter any neighbourhood of $e$, then $(\xi_H)_H$ shows that $G$ has the quasi-small invariant
neighbourhood property (QSIN) of Losert and Rindler \cite{losertr}.  Notice that this occurs in Suzuki's example,
above.

{\bf (ii)} 
%
It is not the case that [Am$^\text{IN}]$-groups are generally inner amenable.  In particular,
the class of inner amenable groups is not closed under extension.

Breuillard and Gelander \cite{breuillardg} provide a dense free group $F_6$ of rank $6$ in $S=\SL_2(\Ree)$.
Let $G=\Ree^2\rtimes F_6$.  If $M$ were a conjugation-invariant mean on $\bl^\infty(G)$, then
it would restrict to one on $\ell^\infty(F_6)\cong\bl^\infty (G/\Ree^2)$.  An observation of Effros \cite{effros}
(see, also \cite{losertr}) shows that the only such mean on $\ell^\infty(F_6)$ is evaluation at $e$.
Hence $M$ is concentrated on the identity component $\bl^\infty(\Ree^2)$, hence giving an $F_6$-invariant
mean on the latter.  

A standard procedure, akin to that of building nets satisfying Reiter's property (P$_1$), allows us
to construct a net $(f_\alp)$ of probablilities in $\bl^1(\Ree^2)$ for which
\begin{equation}\label{eq:reiter}
	\|f_\alp\circ\sig-f_\alp\|_1\rightarrow 0\text{ for each }\sig\text{ in }F_6.
\end{equation}
Since the action of $\SL_2(\Ree)$ is continuous on $\Ree^2$, it is continuous on $\bl^1(\Ree^2)$, and hence,
by density of $F_6$ in $S$, (\ref{eq:reiter}) holds for $\sig$ in $S$.

Since the nilpotent part $N$ of the Iwasawa decomposition $S=KAN$ is the stabilizer of 
${\tiny \begin{bmatrix} 1 \\ 0\end{bmatrix}}$ in $S$, we obtain an isomorphism of $S$-spaces
\begin{equation}\label{eq:Siso}
	k(\theta)a(t)N\mapsto k(\theta)a(t)\begin{bmatrix} 1 \\ 0\end{bmatrix}
	=e^t\begin{bmatrix} \cos\theta \\ \sin\theta \end{bmatrix}:S/N\to\Ree^2\setminus \{0\}
\end{equation}
where $k(\theta)$ is the usual rotation by angle $\theta$ and $a(t)$ diagonal matrix with entries
$e^t$ and $e^{-t}$.  The usual formula for the Haar integral in $S$ gives invariant integral
$\int_\Ree\int_0^{2\pi}f(k(\theta)a(t)N)e^{2t}\,d\theta\,dt$ on $S/N$, while the Jacobian of
$(\theta,t)\mapsto e^t{\tiny \begin{bmatrix} \cos\theta \\ \sin\theta\end{bmatrix}}$ 
in $\Ree^2\setminus\{0\}$ is $e^{2t}$.  Hence
the identification (\ref{eq:Siso}) identifies the $S$-invariant measure on $S/N$ with the Lebesgue measure
on $\Ree^2\setminus\{0\}$.

Thus the net $(f_\alp)\subset\bl^1(\Ree^2)$, above, induces a  net of means $(m_\alp)$ on the space of
left uniformly continuous functions $\LUC(S/N)$ which is asymptotically invariant for left translation.
Any weak* cluster point of $(m_\alp)$ must be an invariant mean, and hence $G/N$ is 
an amenable homogeneous space in the sense of Eymard \cite{eymard1}.  Since $N$ is itself
amenable, \cite[\S3 1$^\circ$]{eymard1} would imply that $S$ is amenable, which is absurd.  Hence
a mean $M$, as above, cannot exist.
\end{remark}

We wish to thank Jason Crann for pointing out an error in an earlier version of Remark \ref{rmk}.  This also addresses an error in \cite[Rem.\ 2.4]{leess}.

\section{Nuclearity and amenable traces}\label{sec:nuclear}

We wish to give another perspective to the question of existence of a trace.  For unital C*-algebras
the following is well-known.  

\begin{lemma}\label{lem:funcmaxtp}
Let $\fA$ be a non-unital C*-algebra and $\tau$ a tracial state on $\fA$.  Then
the map
\begin{equation}\label{eq:tauontp}
a\otimes b^\op\mapsto \tau(ab)
\end{equation}
extends to a state on the maximal tensor product $\fA\otimes_{\max}\fA^\op$.
\end{lemma}

\begin{proof}
We sketch a minor variant of the Gelfand-Neimark-Segal (GNS) construction.
We let $\fA_1$ denote the C*-unitzation of $\fA$, to which the any state cannonically extends.
We complete the space $\fA_1/\fN$ where $\fN=\{a\in \fA:\tau(a^*a)=\tau(aa^*)=0\}$ with
inner product $\langle a+\fN,b+\fN\rangle=\tau(b^*a)=\tau(ab^*)$, to get $\fH$, and let $\pi:\fA\to\fB(\fH)$ and
$\pi^\op:\fA^\op\to\fB(\fH)$ be the respective left and right regular representations on $\fA_1/\fN$, i.e.\ on $\fH$.
Then $\pi$ and $\pi^\op$ have commuting ranges, and we see that
\[
\tau(ab)=\langle \pi(a)\pi^\op(b^\op)(1+\fN),1+\fN\rangle
\]
which gives the result.
\end{proof}

In the spirit of Brown \cite[3.1.6]{brown} (really Kirchberg \cite[Prop.\ 3.2]{kirchberg}) we shall say that
a tracial state $\tau$ on $\fA$ is {\it amenable} (or {\it liftable}) provided that (\ref{eq:tauontp})
extends to a state on the minimal tensor product $\fA\otimes_{\min}\fA^\op$.

The following is well-known for discrete groups.

\begin{theorem}\label{theo:amenabletrace}
For a locally compact group $G$, $\cstarr(G)$ admits an amenable tracial state if and only if $G$ is amenable.
\end{theorem}

\begin{proof}
If $G$ is amenable, then it is well known that $\cstarr(G)$ is nuclear 
(see for example \cite[IV.3.5.2]{blackadar}) and hence any trace is amenable.
As $G$ is amenable $1_G\in\trstr(G)$.

To see necessity
let us first recall the well-known fact that $\cstarr(G)$ is {\it symmetric}.  The map
$f\mapsto\tilde{f}$ on $\bl^1(G)$, where $\tilde{f}(t)=\frac{1}{\Del(t)}f(t^{-1})$ for a.e.\ $t$, is
isometric and anti-multiplicative.  The map $\lam(f)\mapsto\lam(\tilde{f})^\op$ extends to an
isomorphism between $\cstarr(G)$ and $\cstarr(G)^\op$.  Indeed, for $\xi,\eta$ in $\bl^2(G)$
we have
\[
\langle \lam(\tilde{f})\xi,\eta\rangle =\langle \xi,\lam(\bar{f})\eta\rangle=\langle \lam(f)\bar{\eta},\bar{\xi}\rangle
\]
and we take the supremum of the absolute value of this quantity over all choices of $\|\xi\|_2=\|\eta\|_2=1$
to see that $\|\lam(\tilde{f})\|=\|\lam(f)\|$.  

Hence if $\cstarr(G)$ admits an amenable tracial state $\tau$, with associated $u$ in $\trstr(G)$,
then
\[
\lam(f)\otimes\lam(g)\mapsto
\lam(f)\otimes\lam(\tilde{g})^\op\mapsto \tau(\lam(f)\lam(\tilde{g}))=\int_G  f\ast\tilde{g}\,u
\]
defines an element of the dual of $\cstarr(G\times G)\cong \cstarr(G)\otimes_{\min}\cstarr(G)$.
We note that
\begin{align*}
\int_G  f\ast\tilde{g}\,u&=\int_G\int_G f(t)\tilde{g}(t^{-1}s)\, dt\, u(s)\,ds \\
&=\int_G\int_G f(t)\frac{g(s^{-1})}{\Del(s)}u(ts)\,ds\,dt
=\int_G\int_Gf(t)g(s)u(ts^{-1})\,ds\,dt
\end{align*}
so $(s,t)\mapsto u(ts^{-1})$ defines an element of $\fsal_r(G\times G)$.  But then
Lemma \ref{lem:restriction} (i) provides that $1\in \fsal_r(G\times G)|_D\subseteq \fsal_r(D)$,
where $D=\{(s,s):s\in G\}\cong G$, so $G$ is amenable.
\end{proof}

The following is part of the main result of Ng \cite{ng}; we use exactly his observations for sufficiency,
but offer a different proof for necessity.

\begin{corollary}\label{cor:nuclear}
A locally compact group is amenable if and only if $\cstarr(G)$ is nuclear and admits a tracial state.
\end{corollary}

\begin{proof}
Necessity is proved in the  sufficiency condition of the theorem above.
It is evident that any tracial state on a nuclear C*-algebra  is amenable, which gives sufficiency.
\end{proof}

\begin{example}
Let $\alg$ be an algebraic linear group which is  
semisimple, e.g.\ $\alg=\mathrm{SL}_n$, $n\geq 2$,
so $\alg(\Que_p)$ is non-amenble.  However, $\alg(\Que_p)$ is type I by a result of Bern\v{s}te\v{\i}n
\cite{bernstein}
and hence $\cstarr(\alg(\Que_p))$ is nuclear (see, for example, \cite[IV.3.31 \& IV.3.35]{blackadar}).  Thus thanks to the
last corollary, $\cstarr(\alg(\Que_p))$ admits no tracial state.

The field $\Que_p$ of $p$-adic 
numbers is topologically singly generated as a ring:  $\Que_p=\langle 1/p\rangle$.  Hence
$\alg(\Que)$ is compactly generated.  Hence it follows Theorem \ref{theo:main}
that this group is not in the class [Am$^{\text{IN}}$].
\end{example}





\end{document}